\documentclass[12pt]{amsart}
\usepackage{mystyle,mymacros}
\usepackage[all]{xy} 		
\usepackage{longtable}
\usepackage{tikz}
\usepackage{amsxtra,amsmath,amsfonts,amscd, amssymb, mathrsfs}

\usepackage{color}
 
 \newcommand{\eric}[1]{{\color{red} \sf $\clubsuit\clubsuit\clubsuit$ Eric: [#1]}}
 
  \newcommand{\taylor}[1]{{\color{purple} \sf $\diamondsuit\diamondsuit\diamondsuit$ Taylor: [#1]}}

\makeatletter
\@namedef{subjclassname@2010}{%
 \textup{2010} Mathematics Subject Classification}
\makeatother

\title[Total $p$-differentials]{Total $p$-differentials on schemes over $\Z/p^2$}

\author{Taylor Dupuy}
\address{Taylor Dupuy, Department of Mathematics \& Statistics, University of Vermont, Burlington, VT 05405, USA}
\email{taylor.dupuy@uvm.edu}

\author{Eric Katz}
\address{Eric Katz, Department of Mathematics, The Ohio State University University, 231 W. 18th Ave., Columbus, OH 43210, USA}
\email{katz.60@osu.edu}

\author{Joseph Rabinoff}
\address{Joseph Rabinoff, School of Mathematics, Georgia Institute of Technology, Atlanta, GA 30332-0160, USA}
\email{jrabinoff@math.gatech.edu}

\author{David Zureick-Brown}
\address{David Zureick-Brown, Dept. of Math and CS, Emory University, 400 Dowman
 Dr., W401, Atlanta, GA 30322, USA}
\email{dzb@mathcs.emory.edu}

\newcommand{\defi}[1]{\textbf{#1}} 				

\def\presuper#1#2%
  {\mathop{}%
   \mathopen{\vphantom{#2}}^{#1}%
   \kern-\scriptspace%
   #2}

\newcommand{\DI}{\operatorname{DI}}
\newcommand{\FKS}{\operatorname{FKS}}
\newcommand{\OO}{\mathcal{O}}
\newcommand{\Def}{\operatorname{Def}}
\newcommand{\id}{\operatorname{id}}
\newcommand{\FDer}{\operatorname{FDer}}

\begin{document}

\begin{abstract}
For a scheme $X$ defined over the length $2$ $p$-typical Witt vectors $W_2(\bk)$ of a characteristic $p$ field, we introduce total $p$-differentials which interpolate between Frobenius-twisted differentials and Buium's $p$-differentials. They form a sheaf over the reduction $X_0$, and behave as if they were the sheaf of differentials of $X$ over a deeper base below $W_2(\bk)$. This allows us to construct the analogues of Gauss--Manin connections and Kodaira--Spencer classes as in the Katz--Oda formalism.  We make connections to Frobenius lifts, Borger--Weiland's biring formalism, and Deligne--Illusie classes.
\end{abstract}


\maketitle

\section{Introduction}

In a series of papers, Alexandru Buium introduced a notion of arithmetic derivations, extending the analogy between function fields like $\C[[t]]$ and mixed characteristic local fields like $\Z_p^{\unr}$  (see the textbook \cite{buium:2005}). For example, there is an arithmetic derivation operator on $\Z_p^{\unr}$, called a \defi{$p$-derivation,} defined by
\[\delta x=\frac{\phi(x)-x^p}{p}, \]
where $\phi$ is the Frobenius. While $\delta$ does not obey the usual Leibniz rule and is not even additive, it has many formal properties in common with derivations. Just as the usual notion of derivation measures non-constancy of an object over a base, this arithmetic derivation measures how much the Frobenius action on an object over $\Z_p^{\unr}$ differs from $x\mapsto x^p$:  the arithmetic analogue of triviality of a family is the existence of a lift of Frobenius. This is consistent with Borger's approach to the field of one element \cite{Borger:lambdarings}, where one views lifts of Frobenius as a necessary condition for descent to $\F_1$, just as a vector field on a family extending that on the base is a necessary condition for descent from $\C[[t]]$ to $\C$.

More generally, Buium~\cite{Buium1997} classified all arithmetic analogues of derivations on a local ring $R$ of characteristic zero.  By an ``arithmetic analogue of a derivation,'' we mean an operation $\theta\colon R \to R$ on a ring $R$, which has a sum rule, a product rule, and which is appropriately normalized to take certain values on 0 and 1. 
Buium proved that for $R$ a local ring of characteristic zero, there can only exist four types of operations up to equivalence: derivations, difference operations, $p$-derivations, and $p$-difference operations where $p$ is some prime. These various notions afford analogous definitions of differentials and jet spaces.

It turns out that operations $\theta$ satisfying the conditions imposed  in \cite{Buium1997} are intimately tied to ring schemes. 
Under this correspondence, usual derivations correspond to the ring scheme of dual numbers, and Buium's $p$-derivations correspond to the ring of $p$-typical Witt vectors of length two. 
In fact, the ring scheme formalism of Borger--Weiland \cite{borger2005plethystic} provides a uniform framework for formulating differentials and  jet spaces as well as working over a base.


One classical construction using differentials that has been missing its $p$-differential analogue is that of taking differentials on the total space. Specifically, if one has a scheme $\pi\colon X\to S$ smooth over a base, one may form the exact sequence
\[\xymatrix{0\ar[r]&\pi^*\Omega^1_S \ar[r]&\Omega^1_X \ar[r]&\Omega^1_{X/S}\ar[r]&0.}\]
This construction is crucial to Katz--Oda's construction of the Gauss--Manin connection \cite{KatzO:differentiation}. The purpose of this paper is to construct an analogue of the restriction of this exact sequence to a closed fiber, at least for the case where the base is $\Spec \Z/p^2$. In fact, we will consider schemes defined over $\Spec \Z/p^2$ (or more generally $W_2(\bk)$) as families over a ``deeper'' base below $\Z/p^2$ and measure their non-triviality. In this sense, our work comes into contact with the philosophy of the field of one element. The geometric counterpart of our construction can be found in \cite{DFR:orderone}. 

This paper accomplishes the goal by introducing a new class of arithmetic derivations from a ring modulo $p^2$ to a ring modulo $p$ that interpolates between Buium's $p$-derivations and ordinary Frobenius semi-linear derivations.  We call these \defi{total $p$-derivations}. They globalize on schemes $X$ defined over $W_2(\bk)$, the $p$-typical Witt vectors length $2$ of a field of characteristic $p$. In fact, their dual notion, total $p$-differentials form a sheaf $\Omega_X^{1,\tot}$ on $X_0$, the reduction of $X$. They fit into an exact sequence
\[\xymatrix{0\ar[r] &\cO_{X_0}\ar[r]^>>>>>\alpha&\Omega_X^{1,\tot}\ar[r]^>>>>>\beta&F_{X_0}^*\Omega^1_{X_0}\ar[r]&0}\] 
where $F_{X_0}^*\Omega^1_{X_0}$ is the pullback by absolute Frobenius of the sheaf of differentials on $X_0$. The appearance of $\cO_{X_0}$ is a reflection of the fact that $\Spec W_2(\bk)$ should be thought of as an object with a one-dimensional cotangent space over a ``deeper'' scheme.
Regarding $p$-differentials as analogous to differentials on a total space, we construct analogues of the Gauss--Manin connection and the Kodaira--Spencer class.  The usual Gauss--Manin connection is a measure of the non-triviality of a local system arising as the cohomology of a family; our analogue measures the existence of lifts of Frobenius.  This gives an alternative view of Deligne--Illusie's deformation-theoretic approach~\cite{DeligneI:rel} to studying lifts of Frobenius.
The arithmetic analogue of the Kodaira--Spencer class via a different construction already appears with applications to diophantine geometry in \cite{buium:1995}. 
We also give a description of total $p$-differentials in terms of the biring formalism. Further discussion of the  arithmetic analogue of Kodaira--Spencer classes is in \cite{Dupuy:kodaira}.



\section{Total $p$-differentials}

\subsection{Total $p$-derivations}

Let $\bk$ be a field of characteristic $p>2$. Let $R=W_2(\bk)$ be the $p$-typical Witt vectors \cite{Serre1968} of length $2$. Let $\phi\colon R\to R$ be the Frobenius homomorphism. 
Let $A$ be a finitely generated flat $R$-algebra with mod $p$ reduction $A_0$. For $r\in R$, we will write $r$ for the element $r\cdot 1\in A$.
Let $C_p(X,Y)$ be the polynomial
\[C_p(X,Y)=\frac{X^p+Y^p-(X+Y)^p}{p}.\]

\begin{defnsub}\label{D:total-derivation} For an $A_0$-module $D_0$, a {\em total $p$-derivation} of $A$ into $D_0$ is a map $\delta\colon A\to D_0$ such that
\begin{enumerate}
\item for $a,b\in A$, $\delta(a+b)=\delta a+\delta b+(C_p(a,b))\delta p$,
\item for $a,b\in A$ $\delta(ab)=(\delta a)b^p+a^p(\delta b)$, and
\item \label{I:compute} for $r\in R$, $\delta r=((\phi(r)-r^p)/p) (\delta p)$.
\end{enumerate}
\end{defnsub}
Here, we treat $D_0$ as an $A$-module via the natural quotient $A\to A_0$. If $\delta p=0$, then the derivation is a {\em Frobenius semi-linear derivation}. If $\delta p=1$, then it is a $p$-derivation. Both of these notions are central to \cite{buium:p-adic-jets}. Total $p$-derivations interpolate between them.

\begin{example}
For each $c\in \Z/p$ we may specify a total $p$-derivation $\theta_{p,c}\colon \Z/p^2 \to \Z/p$ by setting (following item~\ref{I:compute}):
\[\theta_{p,c}(r)=c \left(\frac{r-r^p}{p}\right).\]
So, $\theta_{p,c}(p)=c$. 
Observe that  total $p$-derivations $\Z/p^2 \to \Z/p$ are determined by their value on $p$: 
if $\theta$ and $\theta'$ are two such derivations then $\theta-\theta'$ is a Frobenius semi-linear derivation, and the only Frobenius semi-linear derivation from $\Z/p^2$ to $\Z/p$ is zero. 
\end{example}

\subsection{Witt interpolation}

Another view of total $p$-derivations can be given in terms of a construction interpolating between length $2$ Witt vectors and a twisted infinitesimal thickening. Let $D_0$ be a $\bk$-algebra. For $c\in D_0$, let $U_c(D_0)$ be the ring whose underlying set is $D_0\times D_0$ equipped with the operations
\begin{eqnarray*}
(x_0,x_1)+(y_0,y_1)&=&(x_0+y_0,x_1+y_1+cC_p(x_0,y_0))\\
(x_0,x_1)\cdot (y_0,y_1)&=&(x_0y_0,x_0^py_1+y_0^px_1).
\end{eqnarray*}
It is easily seen that $U_1(D_0)=W_2(D_0)$. On the other hand, $U_0(D_0)$ is an infinitesimal thickening of $D_0$ twisted by Frobenius. The ring $U_c(D_0)$ can be equipped with an $R$-algebra structure (using the fact that $(1,0)$ is the multiplicative identity) by
\[r\cdot (1,0)=\left(r,c\left(\frac{\phi(r)-r^p}{p}\right)\right).\]
Observe that $p\cdot (1,0)=(0,c)$.
For any value of $c$, there is a homomorphism $U_c(B_0)\to B_0$ given by $(x_0,x_1)\mapsto x_0$. The kernel $I$ satisfies $I^2=0$ and $pI=0$.

The proof of the following is straightforward.
\begin{lemma}
For $e\in A_0,$ there is a homomorphism $h_e\colon U_c(D_0)\to U_{ce}(D_0)$ given by $h(x_0,x_1)=(x_0,ex_1)$.
\end{lemma}

\begin{proposition} 
Given a homomorphism $f\colon A_0\to D_0$ making $D_0$ into an $A_0$-module, total $p$-derivations $\delta$ of $A$ into $D_0$ with $\delta p=c\in D_0$ are in one-to-one correspondence with $R$-algebra homomorphisms
\[(f,\delta)\colon A\to U_c(D_0), \quad a\mapsto (f(\pi_0(a)),\delta a)\]
where $\pi_0\colon A\to A_0$ is reduction.
\end{proposition}

\begin{lemma} Let $g\colon A\to B$ be formally smooth. Let $D_0$ be a $B_0$-algebra, with a $A_0$-algebra structure induced by $g$.
Then a total $p$-derivation $\delta_A$ of $A$ into $D_0$ has a lift $\delta_B$ of $B$ into $D_0$ (that is $\delta_A=\delta_B\circ g$). If, in addition, $g$ is formally unramified, this lift is unique.
\end{lemma}

\begin{proof}
Let $c=\delta_A(p)$. Then, the total $p$-derivativation $\delta_A$ induces a homomorphism $(f,\delta_A)\colon A\to U_c(D_0)$. Consider the commutative diagram
\[\xymatrix{
B\ar[r]^h\ar@{-->}[dr]&D_0\\
A\ar[u]^g\ar[r]&U_c(D_0).\ar[u]
}\]
The top horizontal arrow $h$ is the $B_0$-algebra structure, and the right vertical arrow is reduction. The dotted arrow exists by formal smoothness (and is unique if $f$ is formally unramified). It must be of the form $(h,\delta_B)$ for a total $p$-derivation $\delta_B$.
\end{proof}

The construction of $U_c$ is functorial: if $h\colon B_0\to C_0$ is a $A_0$-algebra homomorphism, then there is an induced homomorphism $h\colon U_c(B_0)\to U_{h(c)}(C_0)$. 

\subsection{Total $p$-differentials}

The module of {\em total $p$-differentials} $\Omega^{1,\tot}_A$ is the $A_0$-module generated by the symbols $d^{\tot}x$ for $x\in A$ subject to the relations
\begin{enumerate}
\item for $a,b\in A$, $d^{\tot}(a+b)=d^{\tot}a+d^{\tot}b+(C_p(a,b))d^{\tot}p$,
\item for $a,b\in A$, $d^{\tot}(ab)=(d^{\tot}a)b^p+a^p(d^{\tot}b)$, and
\item for $r\in R$,  $d^{\tot}r=((\phi(r)-r^p)/p) (d^{\tot} p).$
\end{enumerate}
Here, we treat $\Omega^{1,\tot}_A$ as an $A$-module with the $A$-action factoring through the reduction $A\to A_0$
There is a natural total $p$-derivation $\delta\colon A\to \Omega^{1,\tot}_A$ that satisfies the following universality property: for any total $p$-derivation $\delta'\colon A\to D_0$, there is a unique $A_0$-module homomorphism $g\colon \Omega_A^{1,\tot}\to D_0$ such that $\delta'=g\circ \delta$. 

\begin{example} \label{ex:affinespace} For $A=R[x_1,\dots,x_n]$, a free polynomial algebra, $\Omega_A^{1,\tot}$ is a free $A_0$-module of rank $n+1$, generated by $d^{\tot} p,d^{\tot}x_1,\dots,d^{\tot}x_n$.
\end{example}



There are pullbacks for total $p$-differentials: given an $R$-algebra homomorphism $f\colon A\to B$, there is a well-defined homomorphism of $B_0$-modules, 
\[\Omega_A^{1,\tot}\otimes B_0\to \Omega_B^{1,\tot}\]
taking $d^{\tot}a\otimes b_0\mapsto b_0d^{\tot}(f(a)).$

For an $A_0$-module $M_0$, we will consider $F_{A_0}^*M$ to be the Frobenius tensor product of $M$: $F_{A_0}^*M_0$ is $M_0\otimes_{A_0} A_0$ where $A_0$ is given an $A_0$-action twisted by Frobenius (for $r\in A_0, x\in A_0$, $r\cdot x=r^px.$)
The quotient of $\Omega_A^{1,\tot}$ by $d^{\tot}p$ is $F_{A_0}^*\Omega^1_{A_0}$.
There is a natural $A_0$-module homomorphism $\alpha\colon A_0\to \Omega_A^{1,\tot}$ given by $\alpha(x)=x d^{\tot} p$. 
 A left-inverse of $\alpha$ is a homomorphism $h\colon \Omega_A^{1,\tot}\to A_0$ with $h(d^{\tot}p)=1$. 
 
We will relate left-inverses of $\alpha$ to lifts of the absolute Frobenius. Here, a lift of absolute Frobenius is a ring homomorphism $\phi\colon A\to A$ such that $\phi$ restricts to the Frobenius on $R$ and makes the following diagram commute
\[\xymatrix{
A\ar[r]^\phi\ar[d]_{\pi_0}&A\ar[d]^{\pi_0}\\
A_0\ar[r]^{x\mapsto x^p}& A_0
}\]

\begin{remark}
In much of the literature, one lifts the relative Frobenius which is a $\bk$-linear homomorphism $A_0\to A_0^{(p)}$. Here, $A_0^{(p)}=A_0\otimes_{\bk}\bk$ (with the $\bk$-algebra structure on $\bk$ being given by $r\cdot x=r^p x$) so for $c\in\bk$, $cx\otimes 1=x\otimes c^p$. Lifts of relative Frobenius are in bijective correspondence with lifts of absolute Frobenius. Indeed, one has a commutative diagram
\[\xymatrix{
A_0&\ar[l]_>>>>>>F A_0^{(p)}&\ar[l]_>>>>>G\ar@/_1.5pc/[ll]_{F_{A_0}} A_0\\
& \bk\ar[u]&\ar[l]_{x^p\mapsfrom x}\bk\ar[u]
}\]
where $F$ is the relative Frobenius given by $x\otimes c\mapsto cx^p$ and $G$ is the structure homomorphism given by $x\mapsto x\otimes 1$. Similarly, we may form $A^{(\phi)}=A\otimes_R R$ where $R$ is given an $R$-algebra structure by $r\cdot x=\phi(r)x$. A lift of relative Frobenius is an $R$-homomorphism $\Phi\colon A\to A^{(\phi)}$ reducing to $F$.  Throughout this paper, except in remarks, we will use absolute Frobenius which will be denoted by subscripted $F$. 
\end{remark}

\begin{lemma} \label{l:locallifts} Left-inverses $h\colon \Omega_A^{1,\tot}\to A_0$ of $\alpha$ are in bijective correspondence with lifts of Frobenius on $A$ where the correspondence is given by
\[h\mapsto [\phi_h\colon x\mapsto x^p+ph(d^{\tot}x)].\]
\end{lemma}

\begin{proof}
Let $h$ be a splitting. Define $\phi_h\colon A\to A$ by $\phi_h(x)=x^p+ph(d^{\tot}x)$. We observe that
\begin{eqnarray*}
\phi_h(x+y)&=&(x+y)^p+ph(d^{\tot}(x+y))\\
&=&(x+y)^p+ph(d^{\tot}x+d^{\tot}y+C_p(x,y)d^{\tot}p)\\
&=&x^p+y^p+ph(d^{\tot}x+d^{\tot}y)\\
&=&\phi_h(x)+\phi_h(y)
\end{eqnarray*}
and that
\begin{eqnarray*}
\phi_h(xy)&=&(xy)^p+ph(d^{\tot}(xy))\\
&=&(xy)^p+ph(x^p d^{\tot}y+y^p d^{\tot}x)\\
&=&(x^p+ph(d^{\tot}x))(y^p+ph(d^{\tot}y))\\
&=&\phi_h(x)\phi_h(y).
\end{eqnarray*}

If $\phi\colon A\to A$ is a lift of Frobenius, then we define $h\colon A\to A_0$ by $h(x)=(\phi(x)-x^p)/p$ (which is well-defined because $A$ is a flat $R$-algebra). We claim that $h$ factors through $d^{\tot}\colon A\to \Omega^{1,\tot}$. This follows from reversing the above arguments and noting that $\phi$ restricts to the usual Frobenius on $R$. \end{proof}

\begin{lemma} \label{l:fundamental} Let $A$ be a smooth $R$-algebra. Then, 
there is an exact sequence of $A_0$-modules,
\[\xymatrix{
0\ar[r]&A_0\ar[r]^>>>>>\alpha&\Omega_A^{1,\tot}\ar[r]^>>>>>\beta&F_{A_0}^*\Omega^1_{A_0}\ar[r]&0.
}\]
\end{lemma}

\begin{proof}
We only need to show that $\alpha$ is an injection. Because we are in the affine setting, there is a lift of Frobenius on $A$ by smoothness, and therefore a left-inverse of $\alpha$.
\end{proof}

\subsection{Globalization}

The above constructions globalize by gluing. 
Recall that on an open affine $U_0=\Spec A_0\subset X_0$, if the quasicoherent sheaf $\cF$  corresponds to a module $M$, the pullback by absolute Frobenius $F_{X_0}^*\cF$ on $U$ corresponds to $F_{A_0}^*M$.
Consequently, for $X$ smooth over $R$, the short exact sequence of Lemma~\ref{l:fundamental} globalizes to a sequence of sheaves of $\cO_{X_0}$-modules, the {\em fundamental exact sequence of total $p$-differentials}:
\begin{equation}
\xymatrix{
0\ar[r] &\cO_{X_0}\ar[r]^>>>>>\alpha&\Omega_X^{1,\tot}\ar[r]^\beta&F_{X_0}^*\Omega^1_{X_0}\ar[r]&0}
\label{e:fundamental}
\end{equation}
This exact sequence is an imprecise analogue of the first exact sequence of K\"{a}hler differentials \cite[Prop~8.3A]{Hartshorne:AG}, albeit with the differentials twisted by Frobenius. If we were to imagine that there is a sequence of projections $X\to \Spec R \to\Spec \bk$: the non-existent $``\Omega^1_{R/\bk}"$ would be the sheaf $R_0$ which pulls back to $\cO_{X_0}$.

The existence of the locally split short exact sequence immediately yields the following.
\begin{proposition} If $X$ is smooth over $R$, then $\Omega_X^{1,\tot}$ is locally free.
\end{proposition}

Given a morphism $f\colon X\to Y$ over $\Spec R$, there is a morphism of sheaves on $X_0$, $f^*\Omega^{1,\tot}_Y\to \Omega^{1,\tot}_X$.

\begin{lemma} \label{lem:etale} Let $f\colon X\to Y$ be a smooth morphism of smooth schemes over $R$. Then $f$ induces an injection $f^*\Omega_Y^{1,\tot}\to \Omega_X^{1,\tot}$. If, in addition, $f$ is \'{e}tale, then $f$ is an isomorphism.
\end{lemma}

\begin{proof}
We have a commutative diagram of locally free sheaves with exact rows:
\[\xymatrix{
0\ar[r] &\cO_{X_0}\ar[r]^\alpha&\Omega_X^{1,\tot}\ar[r]^\beta&F_{X_0}^*\Omega^1_{X_0}\ar[r]&0\\
0\ar[r] &f^*\cO_{Y_0}\ar[r]^\alpha\ar[u]&f^*\Omega_Y^{1,\tot}\ar[r]^\beta\ar[u]&f^*F_{Y_0}^*\Omega^1_{Y_0}\ar[r]\ar[u]&0.}\]
The left vertical arrow is an isomorphism while the right vertical arrow is injective (and an isomorphism if $f$ is \'{e}tale).
\end{proof}

By globalizing Lemma~\ref{l:locallifts}, we obtain the following:

\begin{proposition} Splittings of the fundamental exact sequence (\ref{e:fundamental}) are in bijective correspondence with lifts of Frobenius. The set of lifts of Frobenius is a torsor over $H^0(X_0,\HHom(F_{X_0}^*\Omega^1_{X_0},\cO_{X_0})).$
\end{proposition}

\begin{proof}
A splitting of the fundamental exact sequence locally gives lifts of Frobenius which glue into a global lift.
Any two splittings differ by a global section of $\HHom(F_{X_0}^*\Omega^1_{X_0},\cO_{X_0})$. \end{proof}

\subsection{Total jet spaces}

Total $p$-differentials were originally developed by the authors to recontextualize some arguments of Buium \cite{buium:p-adic-jets}. Specifically, given a scheme $X$ over $W(\bk)$, Buium introduced the $p$-jet space $X^1$. Its reduction to $\bk$, $X_0^1$ is a torsor for the Frobenius tangent space (that is, the pullback of the tangent space of $X_0$ by absolute Frobenius) and is related to the Greenberg transform. In \cite[Prop~1.10]{buium:p-adic-jets}, Buium considers the projective completion of $X_0^1$ (considered as an affine torsor). By making use of the fundamental exact sequence,
one can see that the projectivization of $\Omega_X^{1,\tot}$ is manifestly a projective completion of the dual to $F_{X_0}^*\Omega^1_{X_0}$. This gives a useful functorial understanding of such completions.

%

\subsection{Birings}

In  \cite{borger2005plethystic}, Borger and Weiland make use of the formalism of birings, which are ring objects in the category of rings, to put certain universal constructions in algebra on equal footing. One has two subcategories of rings $\cR_1,\cR_2$ and wishes to study functors $F\colon \cR_1\to \cR_2$ between them. The example to keep in mind is the formation of Witt vectors from the category of $\Z/p$-algebras to the category of $\Z/p^2$-algebras. In certain cases, the functor will be representable by a biring. That is, there is a ring $Q$ equipped with the data of coaddition $\Delta^+\colon Q\to Q\otimes Q$ and comultiplication $\Delta^{\times}\colon Q\to Q\otimes Q$ along with additive and multiplicative counits and additive antipodes such that for $A\in\cR_1$,
\[F(A)=\Hom(Q,A)\]
where $\Hom$ denotes ring homomorphisms. The coaddition and comultiplication induce addition and multiplication on $F(A)$.
One may further consider categories $\cR_i$ of algebras over fields $\bk_i$ and then incorporate the data for algebra structures.

We will describe the construction in terms of $\Z/p$ and $\Z/p^2$ knowing that they can be replaced by $\bk$ and $W_2(\bk)$ for a characteristic $p$ field $k$. For $c\in \Z/p$, we will consider the Witt interpolation $U_c$ as a functor from $\Z/p$-algebras to $\Z/p^2$-algebras.
We define 
\[ Q_c = (\Z/p)[e,\eta] \]
with the rules 
 \begin{eqnarray*}
 \Delta^+(e)&=& e\otimes 1+1\otimes e\\
 \Delta^+(\eta) &=& \eta \otimes 1 + 1 \otimes \eta - c \sum_{j=1}^{p-1} \frac{1}{p} {p \choose j} e^{p-j}\otimes e^j \\
 \Delta^{\times}(e)&=&e\otimes e\\
 \Delta^{\times}(\eta) &=& e^p\otimes \eta + \eta \otimes e^p 
 \end{eqnarray*}
 with the other structures defined naturally. It is straightforward verification to show the following:
 \begin{lemma}
 For any $k=\Z/p$-algebra $D_0$, there is an isomorphism:
 \begin{eqnarray*}
\Hom_{\Z/p}(Q_c,D_0)&\rightarrow&U_c(D_0)\\
f&\mapsto&(x_0,x_1)=(f(e),f(\eta)).
\end{eqnarray*}
\end{lemma}

The image of $f$ under this map (in the case $c=1$) is easily seen to be the first two Witt components. Furthermore, we may implement the $\Z/p^2$-algebra structure on $U_c(D_0)$ by the map
\begin{eqnarray*}
\beta\colon \Z/p^2&\to&\Hom_{Z/p^2}(Q,D_0)\\
a&\mapsto&(a,\theta_c(a))
\end{eqnarray*}
(where $\theta_c\colon\Z/p^2\to \Z/p$ is the unique total $p$-derivation with $\theta_c(p)=c$)
so that for $e\in\Z/p^2$, $e\cdot x=(\beta(e))x$ where the product is taken in $U_c(D_0)$.
 
There is an enriched tensor product operation $\odot$ that acts as an left adjoint to Homs of algebras so that if $D_0$ is a $\Z/p$-algebra and $B_1$ is a $\Z/p^2$-algebra,
\[\Hom_{\Z/p}(Q_c\odot B_1,D_0)=\Hom_{\Z/p^2}(B_ 1,\Hom_{\Z/p}(Q_c,D_0)).\]
Total $p$-differentials can be constructed in this formalism. One defines $Q_t=(\Z/p)[t,e,\eta]$ using the above biring structure with the indeterminate $t$ substituted for $c$. Then, for a $\Z/p^2$-algebra $B_1$, $Q_c\odot B_1$ is the free symmetric $(B_1\otimes_{\Z/p^2} \Z/p)$-algebra on the total $p$-differentials $\Omega^{1,\tot}_{B_1}$.
 
%
%
%
%
%
%
%
%
%
%
%
%

\section{Extension classes from the fundamental exact sequence}

In this section, we will suppose that $X$ is smooth over $R$. Recall that a splitting of the fundamental exact sequence, if it exists, induces a lift of Frobenius on $X$.  As an extension, the fundamental exact sequence gives an element
\[\kappa\in \Ext^1(F_{X_0}^*\Omega^1_{X_0},\cO_{X_0})=H^1(X_0,F_{X_0}^*TX_0)\]
 (where $TX_0$ is dual to $\Omega^1_{X_0}$) that is the obstruction to a lift of Frobenius. 
We might call this the arithmetic Kodaira--Spencer class. If we imagine $\Spec R$ to be a point with a vector in the ``$p$-direction'', this would be analogous to the Kodaira--Spencer map applied to this vector. The extension class is interpreted as a \v{C}ech class by picking a covering $U_i$ of $X$ such that there is a splitting $\sigma_i\colon F_{(U_i)_0}^*\Omega^1_{(U_i)_0}\to \Omega^{1,\tot}_{U_i}$. Then,  $s_{ij}=\sigma_i-\sigma_j$ gives a homomorphism $F_{(U_i\cap U_j)_0}^*\Omega^1_{(U_i\cap U_j)_0}\to \cO_{(U_i\cap U_j)_0}$. The cocycle $\{s_{ij}\}$ is a representative of $\kappa$ in  $H^1(X_0,\HHom(F_{X_0}^*\Omega_{X_0},\cO_{X_0}))=H^1(X_0,F_{X_0}^*TX_0).$

The Deligne--Illusie class~\cite[Sec.~2]{DeligneI:rel} (see also \cite{Dupuy2017}), which is sometimes treated in the literature as analogous to the Kodaira--Spencer class, also obstructs the lift of Frobenius, and it is indeed equivalent to the above class. We can construct the Deligne--Illusie class as a \v{C}ech class.  Cover $X$ with affine opens $U_i$ which carry lifts of semi-linear Frobenius $F_{U_i}\colon U_i\to U_i$. One defines a $1$-cocycle on $X$ by 
\[
c_{ij}=F^*_{U_i}-F^*_{U_j}\colon  \cO_{U_i\cap U_j}\to pF_*\cO_{U_i\cap U_j}\]
where $F$ is some lift of Frobenius on $U_i\cap U_j$ (the subsheaf $pF_*\cO_{U_i\cap U_j}$ is independent of choices).
Then $c_{ij}\colon \cO(U_i\cap U_j)\to \cO(U_i\cap U_j)$ factors through the differential $d$ in the sense that there is a cocycle 
\[h_{ij}\colon \Omega^1_{X_0}(U_i\cap U_j)\to (F_{X_0})_*\cO_{X_0}(U_i\cap U_j)\]
that fits into a commutative diagram
\[\xymatrix{
\cO_{X}\ar[d]_{\pi_0} \ar[r]^{c_{ij}} & p F_*\cO_X\\
\cO_{X_0}\ar[d]_d&\\
\Omega^1_{X_0}\ar[r]^>>>>>{h_{ij}}\ar[r]&(F_{X_0})_*\cO_{X_0}\ar[uu]_p
}\]
This cocycle $h_{ij}$ gives an element of 
\[H^1(X_0,\HHom(\Omega^1_{X_0},(F_{X_0})_*\cO_{X_0}))=H^1(X_0,\HHom(F_{X_0}^*\Omega_{X_0}^1,\cO_{X_0})).\]

\begin{remark} The Deligne--Illusie class is also defined using lifts of relative Frobenius $F$ and written as a element of $H^1(X^{(p)}_0,\HHom(\Omega^1_{X^{(p)}_0},F_*\cO_{X_0}))$ where $X_0^{(p)}$ is the base-change by Frobenius, $X_0\times_\bk \bk$. To see that  this is equivalent, we relate the  sheaves of differentials that appear. The absolute Frobenius on $X_0$ factors as
\[\xymatrix{
X_0\ar[r]^F&X_0^{(p)}\ar[r]^G&X_0}\]
where $G$ is the isomorphism induced by the base-change by Frobenius on $\bk$.
From the exact sequence 
\[\xymatrix{
0\ar[r]&G^*\Omega^1_{X_0}\ar[r]&\Omega^1_{X_0^{(p)}}\ar[r]&\Omega^1_{X_0^{(p)}/X_0}\ar[r]&0
}\]
and the vanishing of $\Omega^1_{X_0^{(p)}/X_0}$, we have an isomorphism between $G^*\Omega^1_{X_0}$ and $\Omega^1_{X_0^{(p)}}$. Pulling each back by $F$, we get
\[F_{X_0}^*\Omega^1_{X_0}\cong F^*\Omega^1_{X_0^{(p)}}.\]
\end{remark}

\begin{theorem}
We have the equality up to sign of the extension class and the Deligne--Illusie class $\kappa=-h$.
\end{theorem}

\begin{proof}
The Frobenius lift $F_{U_i}^*\colon \cO(U_i)\to \cO(U_i)$ is equivalent to a splitting $h_i\colon \Omega^{1,\tot}_{U_i}\to \cO_{(U_i)_0}$ over $U_i$ by
\[F^*_{U_i}(x)=x^p+h_i(d^{\tot}x).\]
This splitting induces $\sigma_i\colon F_{U_i}^*\Omega^1_{(U_i)_0}\to\Omega^{1,\tot}_{U_i}$ 
by $\sigma_i(F_{U_i}^*dx_0)=d^{\tot} x-\alpha(h_i(d^{\tot} x))$ for any lift $x$ of $x_0$.
Consequently, for $x\in \cO(U_i\cap U_j)$, we have 
\[\alpha(F^*_{U_i}(x)-F^*_{U_j}(x))=\alpha(h_i(d^{\tot}x)-h_j(d^{\tot}x))=\sigma_j(F_{U_j}^*dx_0)-\sigma_i(F_{U_i}^*dx_0)\]
giving the equality of the Deligne--Illusie and the extension classes.
\end{proof}

\section{Arithmetic Gauss--Manin homomorphism}

There is an arithmetic Gauss--Manin homomorphism in our framework. It is connecting homomorphism  $H^0(X_0,F_{X_0}^*\Omega_{X_0}^1)\to H^1(X_0,\cO_{X_0})$ of the long exact sequence attached to 
\[\xymatrix{0\ar[r] &\cO_{X_0}\ar[r]&\Omega_X^{1,\tot}\ar[r]&F_{X_0}^*\Omega^1_{X_0}\ar[r]&0.}\]
This is analogous to the Gauss--Manin connection for the first de Rham cohomology over a one-dimensional base $g\colon Y\to S$ arising from the exact sequence
\[\xymatrix{0\ar[r] &g^*\Omega^1_S\ar[r]&\Omega_Y^1\ar[r]&\Omega^1_{Y/S}\ar[r]&0.}\]

\begin{theorem}
The arithmetic Gauss--Manin homomorphism is given by cup product with the extension class $\kappa$.
\end{theorem}

\begin{proof}
We compute the arithmetic Gauss--Manin homomorphism. Let $\widetilde{\omega}\in H^0(X_0,F_{X_0}^*\Omega_{X_0}^1).$ We pick a covering $U_i$ of $X$ such that  on $U_i$, $\widetilde{\omega}_i$ lifts to 
$\omega_i\in \Omega^{1,\tot}_{U_i}$. Then, the image of the Gauss--Manin homomorphism is the cocycle $g_{ij}=\omega_i-\omega_j$. Because $\widetilde{\omega}_i=\widetilde{\omega}_j$ on $U_i\cap U_j$, $g_{ij}$ is valued in $\cO_{X_0}$.

By refining the cover, we may suppose that we have a splitting of the fundamental exact sequence on each $U_i$, $\sigma_i\colon F_{(U_i)_0}^*\Omega_{(U_i)_0}^{1}\to\Omega_{U_i}^{1,\tot}$. Set $\omega_i=\sigma_i(\widetilde{\omega}_i)$
Then the image of the Gauss--Manin map is given by the cocycle $g_{ij}=\sigma_i(\omega_i)-\sigma_j(\omega_j)$. On $U_i\cap U_j$ we have
\begin{eqnarray*}
g_{ij}&=&\sigma_i(\widetilde{\omega}_i)-\sigma_j(\widetilde{\omega}_j)\\
&=&\sigma_i(\widetilde{\omega}_i)-\sigma_j(\widetilde{\omega}_i)\\
&=&(\sigma_i-\sigma_j)(\widetilde{\omega}_i)\\
&=&s_{ij}(\widetilde{\omega}_i).
\end{eqnarray*}
Therefore, by unwinding the definition of the cup product, we see 
\[\{g_{ij}\}=\{\widetilde{\omega}_i\}\cup\{s_{ij}\}.\]
\end{proof}
\bibliographystyle{amsalpha}
\bibliography{master}

\vspace{.2in}

\end{document}